\documentclass[11pt, reqno]{amsart}
\usepackage[english]{babel}
\usepackage[table]{xcolor}
\usepackage[utf8]{inputenc}
\usepackage{amsmath, amsfonts, amsthm, amscd, amssymb, fullpage, fancyhdr, upgreek, amssymb, hyperref, enumerate, comment, siunitx, enumitem, graphicx, mathrsfs, mathtools, xcolor,
microtype}
\usepackage{listings}
\usepackage[scr]{rsfso}
\usepackage{thmtools}
\usepackage{thm-restate}
\usepackage{amssymb}

\newcolumntype{P}[1]{>{\centering\arraybackslash}p{#1}}

\usepackage[parfill]{parskip}

\newtheorem{theorem}{Theorem}[section]
\newtheorem{proposition}[theorem]{Proposition}

\newtheorem{lemma}[theorem]{Lemma}
\newtheorem*{definition}{Definition}

\newtheorem*{remark}{Remark}

\numberwithin{equation}{section}

\newcommand{\bburl}[1]{\textcolor{blue}{\url{#1}}}

\usepackage{tikz}
\usepackage{tkz-tab}
\usepackage{tkz-graph}
\usetikzlibrary{shapes.geometric,positioning}

\newcommand{\hr}[1]{\href{#1}{\url{#1}}}

\newcommand{\nocontentsline}[3]{}

\newcommand{\zff}{\mathbf{F}_2[x]}
\newcommand{\zffmi}[1]{(\zff/(#1))^{\times}}
\newcommand{\zffmodinv}{\zffmi{r}}
\newcommand{\zffmodinvs}{\zffmi{s}}
\newcommand{\zffmodinvt}{\zffmi{t}}
\newcommand{\betatrunc}{\beta^{\trunc}}

\newcommand\eps{\varepsilon}

\newcommand\ord{\operatorname{ord}}
\newcommand\denom{\operatorname{denom}}

\newcommand\trunc{\operatorname{trunc}}

\newcommand\pprime{\operatorname{prime}}
\newcommand\modd{\operatorname{mod}}

\title{Sárközy's theorem in $\mathbf{F}_2[x]$}
\author{Aleksandra Kowalska}
\thanks{Mathematical Institute, University of Oxford}

\begin{document}

\begin{abstract}
In \cite{original}, B. Green showed that, conditional on GRH, a subset $A \subseteq [N]$ with $|A| \gg_{\eps} N^{\frac{11}{12}+\eps}$ must contain two elements whose difference is $p-1$ for $p$ a prime. We prove an analogous unconditional result for $\mathbf{F}_2[x]$, improving the exponent to $\frac{7}{8}+\eps$.
\end{abstract}

\maketitle

\section{Introduction}
\label{section_introduction}

In this note, we prove the following:

\begin{theorem}
\label{original_1_1}
    Let $A \subset \{f \in \mathbf{F}_2[x]: \deg(f) < N\}$ be such that $A-A$ contains no polynomial $r-1$, for $r$ irreducible. Then $|A| \ll_{\eps} 2^{(\frac{7}{8}+\eps)N}$.
\end{theorem}

The general outline of the argument and proofs of most of the propositions are analogous to the integer case shown in \cite{original}. The main differences between the proof presented here and that for the integers are as follows:
\begin{itemize}
    \item The functions $\Psi, \Psi'$ (defined in (\ref{psi_def}) and (\ref{psi_prime_def})) can be defined in a field of characteristic $2$ with two terms instead of three, which allows us to decrease the exponent from $\frac{11}{12}$ to $\frac{7}{8}$.
    \item There is no need for a smoothing function, which allows us to simplify many arguments.
    \item The proof of Proposition \ref{section_3_1} is unconditional (as GRH has been proved for function fields).
    \item In Proposition \ref{original_3_3} we are able to calculate the exact value of the expression of interest and prove it in a much simpler way.
    \item In the proof of Proposition \ref{original_3_4} we are only dealing with two Fourier series, which allows us to simplify and shorten the argument.
\end{itemize}

\begin{remark}
\normalfont
    A very similar argument to the one presented here can show that, for $q$ a prime power, if $A \subset \{f \in \mathbf{F}_q[x]: \deg(f) < N\}$ such that $A-A+1$ does not contain any monic irreducible polynomial, we must have $|A| \ll_{\eps} q^{(\frac{11}{12}+\eps)N}$ (the same exponent as in \cite{original}). Since the savings in the exponent possible for $\mathbf{F}_2[x]$ were not available in the general case, the paper is focused on the binary function field. Just before we submitted this preprint to the arXiv, we received a preprint of Lott and Fan which, independently, gives the details of exactly the aforementioned result for arbitrary $q$.
\end{remark}

\textbf{Acknowledgments.} I would like to thank my supervisor Ben Green for suggesting the project, for helpful discussions on it, and for his support. 

\section{Fourier analysis in function fields - background}
\label{section_fourier}

In this section, we provide an introduction to Fourier analysis in function fields, which is used throughout the proof.
It is mostly based on \cite{cubics} and \cite{waring_function_fields}. Analogous results hold for any function field $\mathbf{F}_q[x]$, where $q=p^k$ is a prime power (here they are discussed for $q=2$ as this is the form in which we will use them later).

$\mathbf{F}_2[x]$ behaves in many ways similarly to the integers. Its fraction field $\mathbf{K} := \mathbf{F}_2(x)$ is a counterpart of the rational numbers, and the counterpart of the real numbers is
\begin{equation}
\label{k_infty}
    \mathbf{K}_{\infty}=\mathbf{F}_2\Big(\Big(\frac{1}{x}\Big)\Big)=\left\{ \sum_{i=-\infty}^{n} a_ix^i : n \in \mathbf{Z}, a_i \in \mathbf{F}_2 \right\}.
\end{equation}
We note that $\mathbf{K}_{\infty}$ is a field.

For $a \in \mathbf{K}_{\infty},$ $a=\sum_{i=-\infty}^{+\infty} a_ix^i$ let us define $\ord(a)$ to be the largest $n$ such that $a_i \neq 0$.

Now we can define the analogue of the unit circle:
\begin{equation}
\label{t_circle}
    \mathbf{T}:=\mathbf{K}_{\infty}\big/\mathbf{F}_2[x]=\left\{ \alpha \in \mathbf{K}_{\infty} : \ord(\alpha) \leqslant -1 \right\}.
\end{equation}
For $\theta \in \mathbf{K}_{\infty}$, let $\Vert \theta \Vert$ denote the 'fractional' part of $\theta$, that is, $\theta-||\theta|| \in \mathbf{F}_2[x]$ and $\ord(||\theta||) < 0$.

We define a function $e: \mathbf{K}_{\infty} \rightarrow \mathbf{C}$ which has properties analogous to exponentiation in function fields.

\begin{definition}
    For $\alpha=\sum_{i=-\infty}^{n} a_ix^i$ we define $$e(\alpha) := (-1)^{a_{-1}}.$$
\end{definition}

We now define Fourier analysis on $\mathbf{F}_2[x]/(h)$ for $(h)$ an ideal in $\mathbf{F}_2[x]$ generated by a polynomial $h$. To do so, we first show a basic property of $e$.

\begin{lemma}
\label{indicator_mod_r}
    For $h \in \mathbf{F}_2[x]$ and $f \in \mathbf{F}_2[x] / (h)$, we have $$\sum_{g \in \mathbf{F}_2[x] / (h)} e(fgh^{-1})=\begin{cases}
        2^{\deg(h)} & \text{if $f=0$} \\
        0 & \text{otherwise}
    \end{cases}.$$
\end{lemma}

The proof of this lemma (in a general version for $\mathbf{F}_q[x]$) can be found in \cite[Chapter 5.1]{cubics}.

Let us now define the Fourier transform on $\mathbf{F}_2[x]/(h)$.
For $F: \mathbf{F}_2[x]/(h) \rightarrow \mathbf{C}$, let $\widehat{F}: \mathbf{F}_2[x]/(h) \rightarrow \mathbf{C}$ be defined as follows:
\begin{equation}
    \widehat{F}(g) = \frac{1}{2^{\deg (h)}} \sum_{f \in \mathbf{F}_2[x]/(h)}F(f)e(-fgh^{-1}).
\end{equation}

\begin{lemma}
    We have the following inversion formula: $$F(f)=\sum_{g \in \mathbf{F}_2[x]/(h)} \widehat{F}(g)e(fgh^{-1}).$$
\end{lemma}

This lemma can be easily proved by unfolding the definition of $\widehat{F}$ and using Lemma \ref{indicator_mod_r}.

Let us state one more lemma that will be used regularly throughout this note. Its proof can be found in \cite[Chapter 5.1]{cubics}.

\begin{lemma}
\label{exponnents_monic}
    For $\theta \in \mathbf{K}_{\infty}$, we have
    $$\sum_{\substack{f \in \zff: \\ \deg(f) < N}} e(f\theta)=\begin{cases}
        2^N & \text{if } \ord(||\theta||) < -N\\
        0 & \text{otherwise}
    \end{cases}.$$
\end{lemma}

Let us finish the background section with some notation and definitions.

Let us denote $G_N = \{f \in \mathbf{F}_2[x] : \deg f < N\}$.

We call a polynomial `prime' if it is irreducible and non-constant (usually it is also required to be monic, but this does not matter in $\mathbf{F}_2$).

Let $\Lambda': \zff \rightarrow \mathbf{Z}$ be such that $\Lambda'(f)=\deg (f)$ if $f$ is prime and $0$ otherwise.

Let $A \lessapprox B$ mean that $A \ll N^{O(1)} \cdot B$.

We will often write $\gcd(f, g)$ simply as $(f, g)$ for $f, g$ polynomials.

In function fields, we define $\phi$, $\tau$ and $\mu$ analogously to their integer equivalents: $\phi(f)$ is the number of invertible elements of $\mathbf{F}_2[x] / (f)$, $\tau(f)$ is the number of divisors of $f$ and $\mu(f)$ equals $0$ for $f$ not squarefree and $(-1)^{\omega(f)}$, for $\omega(f)$ the number of prime divisors of $f$.

As in the integer case, for $f=\prod_{i} r_i^{\alpha_i}$ ($r_i$ prime) we have $$\phi(f) = \prod_i\big(\deg(r_i^{\alpha_i}) - \deg(r_i^{\alpha_i - 1})\big)$$ and $$\tau(f) = \prod_{i} (\alpha_i + 1).$$

\section{Van der Corput properties}
\label{section_van_der_corput}

To prove Theorem \ref{original_1_1}, we will first phrase it in terms of the van der Corput property of the set $A$ (Theorem \ref{original_2_1}), which is stronger. Before we do so, let us choose the following parameters (for $\eps>0$ arbitrary and fixed):
\begin{equation}
    K=\big(\frac{1}{8}-\eps\big)N, \ \ R=\frac{N}{4}, \ \ Q=\frac{N}{8}.
\end{equation}
Their use is as follows:
\begin{itemize}
    \item We aim to show that Theorem \ref{original_1_1} holds for $A$ of density at least $2^{N-K}$.
    \item $R$ will be used to choose the 'accuracy' (the number of terms in the Fourier expansion) of a function used to approximate $\Lambda'$.
    \item $Q$ will be used to choose the accuracy of a function approximating the number of divisors of a polynomial.
\end{itemize}

\begin{theorem}
\label{original_2_1}
    There exists a function $\Psi: \mathbf{F}_2[x] \rightarrow \mathbf{R}$ such that
    \begin{enumerate}
        \item $\Psi(f)=0$ if $f \notin \{r-1 : r \in G_N, \ r \pprime\}$
        \item For any $\theta \in \mathbf{T}$, we have $\sum_{f \in G_N} \Psi(f) e (f\theta) \geqslant -2^{N-K}$
        \item $\sum_{f \in G_N} \Psi(f) \gg 2^N$
    \end{enumerate}
\end{theorem}

\begin{proof}[Proof that Theorem \ref{original_2_1} implies Theorem \ref{original_1_1}]
    Suppose that such $\Psi$ exists. Then let us define a cosine polynomial-like function
    \begin{equation}
        T(\theta):=\Big( 2^{N-K} + \sum_{f \in G_N} \Psi(f) e(f\theta) \Big) \Big/ \Big( 2^{N-K} + \sum_{f \in G_N} \Psi(f) \Big).
    \end{equation}
    Let us note that $T(0)=1$ and that for any $\theta \in \mathbf{K}_{\infty}$ we have $T(\theta) \geqslant 0$. Moreover, we can write
    \begin{equation}
        T(\theta)=a_0+\sum_{\substack{r \in G_N, \\ r \pprime}} a_{r-1} e \big((r-1)\theta\big)
    \end{equation}
    for $a_i \in \mathbf{R}$ and
    \begin{equation}
        a_0=2^{N-K} \Big/ \Big( 2^{N-K} + \sum_{f \in G_N} \Psi(f) \Big) \ll 2^{N-K} \Big/ \Big( 2^{N-K}+2^N \Big) \ll 2^{-K}=2^{(-\frac{1}{8}+\eps)N}.
    \end{equation}
    Now let us suppose that $A \subset G_N$ is such that $A-A$ does not contain any $r-1$ for $r$ irreducible. Then we have
    $$|A|^2 \leqslant \sum_{f \in G_{N}} \bigg| \sum_{g \in G_N} 1_A(g) e(fg x^{-N}) \bigg|^2 T(f x^{-N})$$ (as $f=0$ contributes $|A|^2$ and the other contributions are non-negative). The RHS equals
    $$2^Na_0|A| + \sum_{\substack{g_1, g_2 \in A, \\ r \in G_N}} a_{r-1}\sum_{f \in G_{N}} e\big(f(g_1-g_2-(r-1))x^{-N}\big).$$
    From the assumption on $A$, $g_1 - g_2 \neq r-1$ for any $g_1, g_2, r$, so (by Lemma \ref{indicator_mod_r}) the last sum vanishes. Hence $|A|^2 \leqslant 2^Na_0|A|$, so $|A| \leqslant 2^N a_0 \ll_{\eps} 2^{( \frac{7}{8} + \eps )N}$, and therefore Theorem \ref{original_2_1} implies Theorem \ref{original_1_1}.
\end{proof}

In the rest of the note we prove Theorem \ref{original_2_1}.

\section{Definitions and outline of the proof of Theorem \ref{original_2_1}}
\label{main_section}

We start with a few definitions. For $r$ prime and $f \in \zff / (r)$, let
\begin{equation}
    \Lambda_r(f) := \begin{cases}
    0 & \text{if } r \mid f\\
    \frac{2^{\deg(r)}}{2^{\deg(r)} - 1} & \text{otherwise}
\end{cases}
\end{equation}
and
\begin{equation}
    \tau^2_r(f) := \begin{cases}
    4 \cdot \frac{2^{\deg(r)}}{2^{\deg(r)} + 3} & \text{if } r \mid f\\
    \frac{2^{\deg(r)}}{2^{\deg(r)} + 3} & \text{otherwise}
\end{cases}.
\end{equation}
Furthermore, let
\begin{equation}
    \widetilde{\Lambda}_Q(f) := \prod_{r \in G_Q} \Lambda_r(f)
\end{equation}
and
\begin{equation}
    \widetilde{H}_Q(f) := \prod_{r \in G_Q} \tau^2_r(f).
\end{equation}

\begin{remark}
\normalfont
    The function $\Lambda_r(f)$ is a weighted (so that its mean value is $1$) indicator function that checks whether $f$ is divisible by $r$. The product $\widetilde{\Lambda}_Q(f)$ is a weighted approximation of an indicator function of prime polynomials.

    The product $\widetilde{H}_Q(f)$ is an approximation of $\tau^2(f)$. The idea behind the approximation is that for $r$ prime, polynomials divisible by $r$ will have 'on average' twice as many divisors as polynomials not divisible by $r$ (as for $r | f$, any divisor $r \nmid g$ of $\frac{f}{r}$ will contribute $g$ and $gr$ to the set of divisors of $f$). 
\end{remark}

We observe that if we write $\Lambda_r, \tau^2_r$ in terms of their Fourier coefficients modulo $r$, we get
\begin{equation}
    \Lambda_r(f) = 1 - \frac{1}{2^{\deg(r)} - 1} \sum_{g \in \zffmodinv} e(fgr^{-1})
\end{equation}
and
\begin{equation}
    \tau^2_r(f)=1 + \frac{3}{2^{\deg(r)} + 3} \sum_{g \in \zffmodinv} e(fgr^{-1}).
\end{equation}
We can easily check that these formulae hold by using Lemma \ref{indicator_mod_r}.

From these, we can also get Fourier expansions for $\widetilde{\Lambda}_Q, \ \widetilde{H}_Q$
\begin{equation}
\label{lambda_tilde}
    \widetilde{\Lambda}_Q(f) = \sum_{s \mid P_Q} \alpha
(s) \sum_{t \in \zffmodinvs} e\Big(\frac{ft}{s}\Big),
\end{equation}
\begin{equation}
\label{h_tilde}
    \widetilde{H}_Q(f) = \sum_{s \mid P_Q} \alpha'
(s) \sum_{t \in \zffmodinvs} e\Big(\frac{ft}{s}\Big),
\end{equation}
where $P_Q$ is a product of all primes of degree at most $Q$ and $\alpha(s), \alpha'(s)$ are $0$ if $s$ is not square-free and otherwise they are defined as:
\begin{equation}
\label{alpha}
    \alpha(s):=\prod_{\substack{r \mid s, \\ r \pprime}} \frac{-1}{2^{\deg(r)} - 1}=\frac{\mu(s)}{\phi(s)},
\end{equation}
\begin{equation}
\label{alpha_prime}
    \alpha'(s):=\prod_{\substack{r \mid s, \\ r \pprime}} \frac{3}{2^{\deg(r)} + 3}.
\end{equation}
Based on these, we define the Fourier-truncated variants of these functions
\begin{equation}
\label{lambda_def}
    \Lambda_Q(f) := \sum_{s \in G_Q} \alpha(s) \sum_{t \in \zffmodinvs} e\Big(\frac{ft}{s}\Big)
\end{equation}
and
\begin{equation}
\label{h_def}
    H_Q(f) := \sum_{s \in G_Q} \alpha'(s) \sum_{t \in \zffmodinvs} e\Big(\frac{ft}{s}\Big).
\end{equation}
We note that $|\alpha(s)|, |\alpha'(s)| \leqslant \frac{\tau(s)^2}{2^{\deg(s)}}$ (where $\tau(s)$ is the number of divisors of $s$).

Finally, let us define the function $\Psi$ which we will show has the desired properties from Theorem \ref{original_2_1} as follows:
\begin{equation}
\label{psi_def}
    \Psi(f) := \Lambda'(f+1)H_Q(f) \cdot 1_{f \in G_N}.
\end{equation}
Let us also define
\begin{equation}
\label{psi_prime_def}
    \Psi'(f)=\Lambda_R(f+1)H_Q(f) \cdot 1_{f \in G_N}.
\end{equation}
Moreover, let $\beta^{\trunc}, \beta$ be such that
\begin{equation}
\label{beta_trunc_def}
    \Lambda_R(f+1)H_Q(f)=\sum_{\lambda \in \mathbf{F}_2(x) / \mathbf{F}_2[x]}\beta^{\trunc}(\lambda)e(\lambda f)
\end{equation}
and
\begin{equation}
\label{beta_def}
    \widetilde{\Lambda}_R(f+1)\widetilde{H}_R(f)=\sum_{\lambda \in \mathbf{F}_2(x) / \mathbf{F}_2[x]} \beta(\lambda)e(\lambda f).
\end{equation}
We note that $\beta, \betatrunc$ are real.

We will show that $\Psi$ satisfies the desired properties in the following steps:

\begin{proposition}
\label{original_3_1}
    For any $\theta \in \mathbf{T}$, we have $$\bigg| \sum_{f \in G_N} \big( \Psi(f) - \Psi'(f) \big) e(f\theta) \bigg| \leqslant \frac{1}{3}2^{N-K}.$$
\end{proposition}

\begin{proposition}
\label{original_3_2}
    Suppose that $$\bigg| \sum_{f \in G_N} \Psi'(f) e(f\theta) \bigg| \geqslant \frac{1}{3}2^{N-K}.$$
    Then there exists a polynomial $s$ of degree $\leqslant (1+\eps)K+O_{\eps}(\log N)$ and a polynomial $u$ such that $\ord(\theta - \frac{u}{s}) < -N$.
\end{proposition}

\begin{proposition}
\label{original_3_3}
    Suppose that $\theta$ is as in Proposition \ref{original_3_2}. Then we have
    $$\sum_{f \in G_N} \Psi'(f) e(f\theta) = \beta^{\trunc}(\frac{u}{s})\cdot 2^N.$$
\end{proposition}

\begin{proposition}
\label{original_3_4}
    Suppose that $\deg (\denom(\lambda)) \leqslant (1+\eps)K+O_{\eps}(\log N)$. Then $$\Big| \betatrunc(\lambda) - \beta(\lambda) \Big| \leqslant \frac{1}{5}2^{-K}.$$
\end{proposition}

\begin{proposition}
\label{original_3_6}
    The coefficients $\beta(\lambda)$ are real and non-negative with $\beta(0) \gg 1$.
\end{proposition}

We note that $\Psi$ clearly satisfies property $(1)$ of Theorem \ref{original_2_1}. In Proposition \ref{original_3_1} we show that to prove that it satisfies $(2)$, it suffices to show that $\Psi'$ does. If $\Psi'$ fails to satisfy $(2)$ for some $\theta$, then by Proposition \ref{original_3_2} $\theta$ is on a 'major arc'. Propositions \ref{original_3_3}, \ref{original_3_4} and \ref{original_3_6} show that for such $\theta$, $\Psi'$ also satisfies $(2)$. Finally, since $\sum_{f \in G_N} \Psi'(f)=\sum_{f \in G_N} \Psi'(f) e(f \cdot 0)$ and $0$ is on the major arc, these propositions show that $\Psi'$ also satisfies $(3)$.

Before proving the above propositions in order (in Sections \ref{section_3_1} - \ref{section_3_6}), we will show a property of the Fourier coefficients of the considered functions, which will later be used multiple times.

\section{Bounds on Fourier coefficients}
\label{section_fourier_bounds}

In the previous section, we defined functions $\widetilde{\Lambda}_Q,$ $\widetilde{H}_Q$, and their Fourier-truncated variants. We noted that for $\alpha, \alpha'$ their respective Fourier coefficients, we have $|\alpha(\frac{r}{s})|, |\alpha'(\frac{r}{s})| \leqslant \frac{\tau(s)^2}{q^{\deg(s)}}$. In this section, we consider a product of two functions whose Fourier coefficients are bounded in this way, and bound its Fourier coefficients analogously. The bound developed in this section will be used later throughout the proof.

\begin{definition}
    Let us write a function $F: \zff \rightarrow \mathbf{R}$ in terms of its Fourier coefficients as: $$F(f)=\sum_{\lambda \in \mathbf{F}_2(x) / \zff} c(\lambda)e(\lambda f).$$ We say that $F \in \mathcal{C}_B(X)$ (for $B \in \mathbf{N}$) if we have $|c(\lambda)| \ll \frac{\tau(\denom(\lambda))^B}{2^{\deg(\denom(\lambda))}}$ for any $\lambda$, and moreover the coefficients $c$ are supported on $\lambda$ with $\mu(\denom(\lambda))^2=1$ and $\deg(\denom(\lambda)) \leqslant X$.
\end{definition}

\begin{lemma}
\label{original_4_2}
    If $F \in \mathcal{C}_{B_1}(X)$, $G \in \mathcal{C}_{B_2}(Y)$, then $H(f):=F(f)G(f) \in X^{O_{B_1, B_2}(1)} \mathcal{C}_{B_1+2B_2+3}(X+Y).$
\end{lemma}

\begin{proof}
    Let $a, b, c$ be the Fourier coefficients of respectively $F, G, H$. The fact that $c(\lambda)$ is supported on $\lambda$ squarefree with $\deg(\denom(\lambda)) \leqslant X + Y$ follows directly from multiplying the two Fourier series out, as the degree of the denominator of $\frac{f_1}{g_1}+\frac{f_2}{g_2}$ is at most $\deg(g_1g_2)=\deg(g_1)+\deg(g_2)$. Hence, we have only the second claim left to prove.

    We have 
    \begin{multline*}
        c\Big(\frac{f}{g}\Big)=\sum_{\substack{g_1 \in G_{X+1},\\ g_2 \in G_{Y+1}}} \sum_{\substack{f_i \in \zff / (g_i): \\ \frac{f_1}{g_1}+\frac{f_2}{g_2} = \frac{f}{g}}} a\Big(\frac{f_1}{g_1}\Big)b\Big(\frac{f_2}{g_2}\Big) \\
        \ll \sum_{\substack{g_1 \in G_{X+1}, \\ \mu(g_1)^2=1}} \frac{\tau(g_1)^{B_1}}{2^{\deg(g_1)}} \sum_{\substack{g_2 \in G_{Y+1}, \ g_2 | gg_1, \\ \mu(g_2)^2=1}} \frac{\tau(g_2)^{B_2}}{2^{\deg(g_2)}} \#\{f_1 \in \zff / (g_1) : \denom\Big(\frac{f}{g}-\frac{f_1}{g_1}\Big)=g_2\}.
    \end{multline*}

    Now we focus on bounding the inner $\#\{f_1 \in \zff / (g_1) : \denom(\frac{f}{g}-\frac{f_1}{g_1})=g_2\}$ (we note that it is $0$ if $g_2 \nmid gg_1$).

    The condition that $\denom(\frac{a}{q}+\frac{b}{r})=s$ is equivalent to $(ar+bq, rq)=\frac{rq}{s}$. Let $(q, r)=d$, $q=dq'$ and $r=dr'$ (for $q' \perp r'$, $d \perp q', r'$ since $q, r$ are squarefree). Hence, $\denom(\frac{a}{q}+\frac{b}{r})=s$ implies $(ar'+bq', q'r'd)=\frac{q'r'd}{s}$. Since $a \perp q', b \perp r'$, we have $(ar'+bq', q'r'd)=(ar'+bq', d)$ and so we must have $q'r'|s$ (so let $s=q'r't$). Hence, $(ar'+bq', d)=\frac{d}{t}$ and so we need $\frac{d}{t} | (ar'+bq')$, which holds if and only if $ar' \equiv bq' (\modd \ \frac{d}{t})$, so $a \equiv bq'(r')^{-1} (\modd \ \frac{d}{t})$ (as $r' \perp d$). Hence, there are at most $2^{\deg(\frac{tq}{d})}=2^{\deg(\frac{s \cdot (q, r)}{r})}$ possible choices of $a \in \zffmi{q}$ satisfying the condition $\denom(\frac{a}{q}+\frac{b}{r})=s$.

    Now we return to bounding $c(\frac{f}{g})$. From the above considerations, we have
    \begin{multline*}
        c\Big(\frac{f}{g}\Big)
    \leqslant \sum_{\substack{g_1 \in G_{X+1}, \\ \mu(g_1)^2=1}} \frac{\tau(g_1)^{B_1}}{2^{\deg(g_1)}} \sum_{\substack{g_2 \in G_{Y+1}, \ g_2 | gg_1, \\ \mu(g_2)^2=1}} \frac{\tau(g_2)^{B_2}}{2^{\deg( g_2 )}} \cdot 2^{\deg(g_2 \cdot (g_1, g)) - \deg( g )}
    \\ \leqslant \frac{\tau(g)^{B_2+1}}{2^{\deg( g )}} \sum_{\substack{g_1 \in G_{X+1}, \\ \mu(g_1)^2=1}} \tau(g_1)^{B_1+B_2+1} \cdot 2^{\deg( (g_1, g) ) - \deg( g_1 )},
    \end{multline*}
    where the second inequality comes from bounding $\tau(g_2) \leqslant \tau(g)\tau(g_1)$ and bounding the number of occurrences of $g_2$ in the inner sum by $\tau(gg_1) \leqslant \tau(g)\tau(g_1)$.
    Hence, it suffices to show that (for $B=B_1+B_2+1$)
    $$\sum_{\substack{h \in G_{X+1}, \\ \mu(h)^2=1}} \tau(h)^{B} \cdot 2^{\deg( (h, g) ) - \deg(h)} \leqslant \tau(g)^{O_B(1)} \cdot X^{O(1)}.$$
    Since $g$ is squarefree, we note that the above is at most
    $$\prod_{\substack{r \in G_{X+1}, \\ r \nmid g, \ r \pprime}} \Big( 1 + \frac{2^B}{2^{\deg( r )}} \Big) \cdot \prod_{\substack{r \in G_{X+1}, \\ r \mid g, \ r \pprime}} ( 1 + 2^B) = \tau(g)^{\log_2 (1 + 2^B)} \cdot \prod_{\substack{r \in G_{X+1}, \\ r \nmid g, \ r \pprime}} \Big( 1 + \frac{2^B}{2^{\deg( r )}} \Big).$$
    From the prime number theorem for polynomials over finite fields (its proof can be found e.g. in \cite{pnt_finite_fields}) we have
    \begin{multline*}
        \prod_{\substack{r \in G_{X+1}, \\ r \pprime}} \Big( 1 + \frac{2^B}{2^{\deg( r )}} \Big)
        \leqslant \prod_{n=1}^X \Big( 1 + \frac{2^B}{2^n} \Big)^{2 \cdot 2^n/n}
        \leqslant \prod_{n=1}^X \bigg( \sum_{k=0}^{\infty} \frac{2^{B \cdot k}}{2^{kn}} \binom{2^{n+1}/n}{k} \bigg) \\
        \ll \prod_{n=2^{B+1}+1}^X \bigg( \sum_{k=0}^{\infty} \frac{2^{B \cdot k}}{2^{kn}} \cdot \frac{2^{k(n+1)}}{n^k} \bigg)
        = \prod_{n=2^{B+1}+1}^X \frac{n}{n-2^{B+1}} \leqslant \frac{X!}{(X-2^{B+1})!}\leqslant X^{2^{B+1}}=X^{O_B(1)},
    \end{multline*}
    which finishes the proof.
\end{proof}

We note that the Fourier coefficients $\alpha, \alpha'$ of $\Lambda_Q(f), H_Q(f)$ satisfy $|\alpha(\lambda)|, |\alpha'(\lambda)| \leqslant \frac{\tau(\denom(\lambda))^2}{2^{\deg( \denom(\lambda) )}}$.

Since $H_Q(f) \in \mathcal{C}_2(Q)$ and $\Lambda_R(f+1) \in \mathcal{C}_2(R)$, Lemma \ref{original_4_2} allows us to deduce that $\Lambda_R(f+1)H_Q(f) \in N^{O(1)} \mathcal{C}_9(R+Q)$.

\section{Proof of Proposition \ref{original_3_1}}
\label{section_3_1}
In this and the remaining sections we prove Propositions \ref{original_3_1}-\ref{original_3_6}, therefore finishing the proof of Theorem \ref{original_2_1}, and so Theorem \ref{original_1_1} (the results from the previous will be useful multiple times in the process).

Proposition \ref{original_3_1} states that for any $\theta \in \mathbf{T}$ (where $\mathbf{T}$ is as defined in (\ref{t_circle})), we have $$\bigg| \sum_{f \in G_N} \big( \Psi(f) - \Psi'(f) \big) e(f\theta) \bigg| \leqslant \frac{1}{3}2^{N-K}$$
(for $\Psi, \Psi'$ defined in \ref{psi_def} and \ref{psi_prime_def} and $K=(\frac{1}{8}-\eps)N$). This allows us to work in the rest of the proof with $\Psi'$ instead of $\Psi$.

First, we will show that it suffices to prove that $\Lambda'(f)$ has Fourier spectrum close to that of $\Lambda_R(f)$.

\begin{lemma}
    If for some $\theta$ and $0<c<1$ we have
    $$\bigg| \sum_{f \in G_N} \big( \Lambda'(f) - \Lambda_R(f) \big) e(f \theta) \bigg| \lessapprox 2^{cN},$$ then $$\bigg| \sum_{f \in G_N} \big( \Psi(f) - \Psi'(f) \big) e(f \theta)\bigg| \lessapprox 2^{cN+Q}.$$
\end{lemma}

\begin{proof}
    Let
    \begin{equation}
        H_Q(f):=\sum_{\substack{\lambda \in \mathbf{T}, \\ \ord(\denom(\lambda)) \leqslant Q}} h(\lambda)e(\lambda f),
    \end{equation}
    where $h(\lambda)$ is supported on $\lambda$ squarefree. We also recall that $h(\lambda) \leqslant \frac{\tau(\denom(\lambda))^2}{2^{\deg( \denom(\lambda)) }}$. Then we have
    \begin{multline*}
        \bigg| \sum_{f \in G_N} \big(\Psi(f) - \Psi'(f)\big)e(f \theta) \bigg| = \bigg|  \sum_{f \in G_N} \big(\Lambda'(f+1)- \Lambda_R(f+1)\big)H_Q(f)e(f \theta) \bigg|\\ \leqslant \sum_{\substack{\lambda \in \mathbf{T}, \\ \ord(\denom(\lambda)) \leqslant Q}}\big|  h(\lambda) \big| \cdot \bigg| \sum_{f \in G_N} \big(\Lambda'(f+1) - \Lambda_R(f+1) \big) e((\theta + \lambda)f) \bigg|\\ \lessapprox 2^{cN} \sum_{\substack{\lambda \in \mathbf{T}, \\ \ord(\denom(\lambda)) \leqslant Q}} \big|h(\lambda)\big|\leqslant 2^{cN}\sum_{s \in G_Q} \frac{\tau(s)^2}{2^{\deg( s )}} \cdot \phi(s) \leqslant 2^{cN}\sum_{s \in G_Q}\tau(s)^2 \lessapprox 2^{cN+Q},
    \end{multline*}
    where the last inequality holds by Lemma \ref{original_A_1}.
\end{proof}

Let us recall that $K=\big(\frac{1}{8}-\eps\big)N$, $R=\frac{N}{4}$ and $Q=\frac{N}{8}$.
As $(N-2Q)+Q=N-K-\eps N$, to prove Proposition \ref{original_3_1}, it suffices to show that for any $\theta \in \mathbf{T}$, we have $$\bigg| \sum_{f \in G_N} \big( \Lambda'(f) - \Lambda_R(f) \big) e(f \theta) \bigg| \lessapprox 2^{N-2Q}.$$
Let us denote $S_1:=\sum_{f \in G_N} \Lambda'(f) e(f \theta)$ and $S_2 :=\sum_{f \in G_N} \Lambda_R(f) e(f \theta)$. We need to show that $\big| S_1 - S_2 \big| \leq 2^{N-2Q}$.

From Dirichlet approximation for polynomials (Lemma \ref{dirichlet}), there exist $u, s \in \mathbf{F}_2[x], \eta \in \mathbf{T}$ such that
\begin{equation}
    \theta=-\frac{u}{s}+\eta, \ \ \deg(s)<\frac{N}{2}-Q, \ \ \ord(\eta)\leqslant -\deg(s)-\frac{N}{2}+Q
\end{equation}
(it will soon be clear why $\frac{N}{2}-Q$ was chosen to bound the degree of $s$).

First, to estimate $S_1$, we quote the following lemma (its proof can be found in \cite[Theorem 5.3, Lemma 7.1]{hayes})

\begin{lemma}
\label{grhlemma}
    For $k \in \mathbf{N}$ and $\theta$ such that $\theta=\frac{u'}{s'}+\eta'$ for $\deg(s') \leqslant k/2$, $u' \in \mathbf{F}_2[x]$ and $\ord(\eta') < -\ord(s')- \lfloor k/2 \rfloor$, we have $$\bigg| \sum_{\substack{f: \ \deg(f)=k}} \Lambda'(f) e(f \theta) - \frac{\mu(s')}{\phi(s')} 2^k \cdot c(k, \eta') \bigg| < k2^{(3k+5)/4}$$ where $$c(k, \eta'):=\begin{cases}
        1 & \text{if } \ord(\eta')<-k-1\\
        (-1) & \text{if } \ord(\eta')=-k-1\\
        0 & \text{otherwise}
    \end{cases}.$$
\end{lemma}

Then we note that
$$S_1=\sum_{f \in G_{N-2Q-1}}\Lambda'(f)e(f\theta)+\sum_{n=N-2Q}^{N-1}\sum_{f: \ \deg(f)=n}\Lambda'(f)e(f\theta).$$ The first term is $\lessapprox 2^{N-2Q}$ and the second term is, by Lemma \ref{grhlemma}, within $O(1)\cdot N^22^{3N/4} \lessapprox 2^{N-2Q}$ from $$\frac{\mu(s)}{\phi(s)}\sum_{n=N-Q-K}^{\min(N, -\ord(\theta))-1}2^nc(n, \eta).$$ We note in passing that applying Lemma \ref{grhlemma} to $n\geqslant N-2Q$ is the reason for earlier bounding $\deg(s)$ by $\frac{N}{2}-Q$. Hence, we have arrived at the estimation
\begin{equation}
\label{equation_1}
    \bigg| \sum_{f \in G_N}\Lambda'(f)e(f\theta) - \frac{\mu(s)}{\phi(s)} \sum_{n < \min(N, -\ord(\eta))}2^n c(n, \eta) \bigg| \lessapprox 2^{N-2Q}.
\end{equation}
We note that
\begin{equation}
\label{equation_2}
    \sum_{n<\min(N, -\ord(\eta))}2^nc(n, \eta)=\begin{cases}
    2^N-1 & \text{if } \ord(\eta) < -N\\
    -1 & \text{otherwise}
\end{cases}.
\end{equation}

To estimate $S_2$, we note that since $$\Lambda_R(f)=\sum_{s \in G_R} \frac{\mu(s)}{\phi(s)} \sum_{u \in \zffmodinvs} e\Big(\frac{uf}{s}\Big),$$ we have $$\sum_{f \in G_N} \Lambda_R(f) e(f \theta)
= 1_{\deg(s) < R} \frac{\mu(s)}{\phi(s)} \sum_{f \in G_N} e(f \eta) + \sum_{t \in G_R} \sum_{\substack{v \in \zffmodinvt, \\ \frac{v}{t} \neq \frac{u}{s}}} \frac{\mu(t)}{\phi(t)} \sum_{f \in G_N} e\Big( \big(\frac{v}{t} - \frac{u}{s} + \eta \big)f \Big).$$
For $\frac{v}{t}\neq \frac{u}{s}$, $\ord(\frac{v}{t}-\frac{u}{s}) \geqslant \ord(s)-\ord(t) \geqslant -R - \ord(s).$ Since $\ord(\eta) \leqslant -\ord(s) - (\frac{N}{2}-Q) < -R - \ord(s)$, so $\ord(\frac{v}{t} - \frac{u}{s} + \eta) \geqslant -R-\ord(s) > -R-\frac{N}{2}+Q>-N$. Hence, by Lemma \ref{exponnents_monic}
$$\sum_{f \in G_N}e\Big( \big(\frac{v}{t} - \frac{u}{s} + \eta \big)f \Big)=0,$$
and so $$S_2=\frac{\mu(s)}{\phi(s)}\cdot 2^N \cdot 1_{\ord(\eta)<-N} \cdot 1_{\deg(s)<R}.$$

From this and the earlier estimations (\ref{equation_1}) and (\ref{equation_2}) of $S_1$, we can easily deduce the desired bound: if $\deg(s)\geqslant R$, then both $S_1$ and $S_2$ are at most $\frac{2^{N}}{\phi(s)} + 2^{N-2Q}\lessapprox 2^{N-\deg(s)} + 2^{N-2Q} \leqslant 2^{N-R}$. Otherwise, if $\deg(s)<R$, the estimations of $S_1$ and $S_2$ differ by $\frac{\mu(s)}{\phi(s)}$, which is negligible.

\section{Proof of Proposition \ref{original_3_2}}
\label{section_3_2}
Let us recall that Proposition \ref{original_3_2} states that if for some $\theta$ we have
\begin{equation}
\label{assumption_3_2}
    \bigg| \sum_{f \in G_N} \Lambda_R(f+1)H_Q(f)  e(f\theta) \bigg| \geqslant \frac{1}{3}2^{N-K}
\end{equation}
then $\theta$ is on a 'major arc', that is, there exists a polynomial $s$ of degree at most $(1+\eps)K+O_{\eps}(\log N)$ and a polynomial $u$ such that $\ord(\theta - \frac{u}{s}) < -N$. The functions $\Lambda_R, H_Q$ are defined respectively in \ref{lambda_def}, \ref{h_def} and we have $K=\big(\frac{1}{8}-\eps\big)N$, $R=\frac{N}{4}$, $Q=\frac{N}{8}$.

Moreover, let us recall from considerations at the end of Section \ref{section_fourier_bounds} that $\betatrunc$ (the Fourier coefficients of $\Lambda_R(f+1)H_Q(f)$) are supported on those $\lambda$ with $\denom(\lambda)$ squarefree and of degree at most $R+Q$, and are bounded by $|\betatrunc(\lambda)| \ll N^{O(1)} \frac{\tau(\denom(\lambda))^9}{2^{\deg(\denom(\lambda))}}$.

By Dirichlet's approximation theorem for polynomials (Lemma \ref{dirichlet}), we can write
\begin{equation}
    \theta=-\frac{u}{s}+\eta, \ \ \deg(s) < R+Q+1, \ \ \ord(\eta) \leqslant -R-Q-1-\deg(s)
\end{equation}
The assumption (\ref{assumption_3_2}) and the triangle inequality imply that
$$N^{O(1)}\sum_{\substack{t \in G_{R+Q},\\ \mu(t)^2=1}} \frac{\tau(t)^9}{2^{\deg( t )}} \sum_{v \in \zffmodinvt} \bigg| \sum_{f \in G_N} e \Big(f\big(\frac{v}{t}-\frac{u}{s}+\eta\big)\Big) \bigg| \geqslant \frac{1}{3}2^{N-K}.$$
For $\frac{v}{t} \neq \frac{u}{s}$, we have $\ord(\frac{v}{t}-\frac{u}{s})\geqslant -\deg(t)-\deg(s)\geqslant-R-Q-\deg(s)$. Since $\ord(\eta)\leqslant -R-Q-1-\deg(s)$, $\ord(\frac{v}{t}-\frac{u}{s}+\eta) \geqslant -R-Q-\deg(s)>-2R-2Q-1>-N$. Hence, by Lemma \ref{exponnents_monic}, the contribution from any term with $\frac{v}{t} \neq \frac{u}{s}$ is $0$, and so the only non-zero contribution comes from $\frac{v}{t}=\frac{u}{s}$. Hence, we have $$N^{O(1)} \cdot\frac{\tau(s)^9 }{2^{\deg( s )}} \bigg| \sum_{f \in G_N} e(f\eta) \bigg| \geqslant \frac{1}{3}2^{N-K}.$$
This in turn implies that $\ord(\eta)<-N$ and that
$2^{\deg(s)}/\tau(s)^9 \leqslant 3 \cdot 2^{K+O(\log N)}$. From the divisor bound for polynomials (Lemma \ref{divisor_bound}), $\tau(s) \ll_{\eps'} 2^{\deg(s)\eps'}$ for any $\eps' > 0$, so $2^{(1-9\eps')\deg(s)}\ll_{\eps'} 2^{K+O(\log N)}$, which implies that $\deg(s) \ll_{\eps} K(1+\eps)+O_{\eps}(\log N)$ if we take $\eps'=\frac{\eps}{9(1+\eps)}$.

\section{Proof of Proposition \ref{original_3_3}}
\label{section_3_3}
Proposition \ref{original_3_3} states that for $\theta$ on a 'major arc' (that is, close to a rational function), $\Psi'(\theta)$ is proportional to the Fourier coefficient of $\Psi'$ for $\theta$. In particular, we need to show that for $\theta=-\frac{u}{s}+\eta$, where $\deg(s) \leqslant K(1+\eps)+O_{\eps}(\log N)$, $u \in \zff$, $\ord(\eta) < -N$ (the minus sign was added for convenience) we have
$$\sum_{f \in G_N} \Psi'(f) e(f\theta) = \beta^{\trunc}(\frac{u}{s})\cdot 2^N$$
(where $\beta^{\trunc}$ was defined in \ref{beta_trunc_def}).
However,
$$\sum_{f \in G_N} \Psi'(f)e(f\theta)= \sum_{\lambda \in \mathbf{T}} \betatrunc(\lambda)\sum_{f \in G_N} e\Big( \big(\lambda-\frac{u}{s}+\eta\big)f\Big).$$
As in the proof of Proposition \ref{original_3_2}, the terms with $\lambda\neq \frac{u}{s}$ are $0$, so the above equals the contribution of $\lambda=\frac{u}{s}$, which is $\betatrunc(\frac{u}{s}) \cdot 2^N$.

\section{Proof of Proposition \ref{original_3_4}}
\label{section_3_4}
Proposition \ref{original_3_4} states that the Fourier coefficients of $\Psi$ and $\Psi'$ are close for arguments with denominators of a small degree. In particular, we need to show that if $\deg (\denom(\lambda)) \leqslant (1+\eps)K+O_{\eps}(\log N)$, then $$\Big| \betatrunc(\lambda) - \beta(\lambda) \Big| \leqslant \frac{1}{5}2^{-K},$$
where $\beta^{\trunc}$ and $\beta$ were defined respectively in \ref{beta_trunc_def} and \ref{beta_def}.

Let us start the proof by recalling what we know about the Fourier coefficients $\beta$ and $\betatrunc$.
First, both $\beta$ and $\betatrunc$ are supported on $\lambda$ squarefree. Moreover, from the Fourier expansions of $\widetilde{\Lambda}_R, \widetilde{H}_R, \Lambda_R, H_Q$ (defined respectively in \ref{lambda_tilde}, \ref{h_tilde}, \ref{lambda_def}, \ref{h_def}) we have
\begin{equation}
    \betatrunc\Big(\frac{b}{r}\Big)=\sum_{\substack{q_1 \in G_R,\\ q_2 \in G_Q}} \alpha(q_1)\alpha'(q_2) \sum_{\substack{a_i \in \zffmi{q_i},\\ \frac{a_1}{q_1}+\frac{a_2}{q_2} \equiv \frac{b}{r} \ (\modd \ \zff)}} e\Big( \frac{a_1}{q_1} \Big),
\end{equation}
where $\alpha$, $\alpha'$ are defined as in \ref{alpha}, \ref{alpha_prime}, and $\alpha(q),\alpha'(q) \leqslant \frac{\tau(q)^2}{2^{\deg( q )}}$. Similarly, for $r \in G_R$,
\begin{equation}
    \beta\Big(\frac{b}{r}\Big)=\sum_{q_1, q_2} \alpha(q_1)\alpha'(q_2) \sum_{\substack{a_i \in \zffmi{q_i},\\ \frac{a_1}{q_1}+\frac{a_2}{q_2} \equiv \frac{b}{r} \ (\modd \ \zff)}} e\Big( \frac{a_1}{q_1} \Big).
\end{equation}
Therefore, to show that $|\betatrunc(\lambda)-\beta(\lambda)| \leqslant \frac{1}{10}2^{-K}$, it suffices to show that $$D:=\bigg| \sum_{\max(\frac{q_1}{R}, \frac{q_2}{Q}) > 1} \alpha(q_1)\alpha'(q_2) \sum_{\substack{a_i \in \zffmi{q_i},\\ \frac{a_1}{q_1}+\frac{a_2}{q_2} \equiv \frac{b}{r} \ (\modd \ \zff)}} e\Big( \frac{a_1}{q_1} \Big) \bigg| \leqslant \frac{1}{10}2^{-K}.$$
The key to bounding $D$ is bounding the inner sum, which we do in the following lemma:

\begin{lemma}
\label{main_lemma_for_34}
    We have, for any $r, q_1, q_2 \in \zff$ squarefree and $b \in \zffmodinvt$
    $$S := \bigg| \sum_{\substack{a_i \in \zffmi{q_i},\\ \frac{a_1}{q_1}+\frac{a_2}{q_2} \equiv \frac{b}{r} \ (\modd \ \zff)}} e\Big( \frac{a_1}{q_1} \Big) \bigg| \ll \tau(r).$$
\end{lemma}

Before proving the lemma, we will show how it implies Proposition \ref{original_3_4}. First, we note that if $a_1, a_2$ such that $\frac{a_1}{q_1}+\frac{a_2}{q_2} \equiv \frac{b}{r} \ (\modd \ \zff)$ exist, we must have $q_1 | rq_2$ and $q_2 | rq_1$.

Suppose that the lemma holds. Then
$$D \leqslant N^{O(1)} \sum_{\substack{\max(\frac{q_1}{R}, \frac{q_2}{Q}) > 1, \\ q_1|rq_2, \\ q_2 | rq_1}} \frac{\tau(q_1)^2\tau(q_2)^2\tau(r)}{2^{\deg( q_1q_2 )}} \lessapprox \tau(r) \sum_{\substack{q_1 \geqslant Q, \\ q_2 \leqslant q_1, \\ q_1|rq_2, \\ q_2 | rq_1}}\frac{\tau(q_1)^2\tau(q_2)^2}{2^{\deg( q_1q_2 )}}.$$
Since $q_2 | rq_1$, $\tau(q_2)\leqslant \tau(r)\tau(q_1)$. Moreover, since $q_1|rq_2$, $\frac{q_1}{(r, q_1)}|q_2$, so $\deg q_2 \geqslant \deg(\frac{q_1}{(r, q_1)} )$. Hence, we have
\begin{multline*}
    D \lessapprox \tau(r)^{3} \sum_{q_1 \geqslant Q} \frac{\tau(q_1)^{4} \cdot2^{\deg ((q_1, r))}}{2^{2\deg( q_1 )}} \leqslant \tau(r)^{3} \sum_{s | r} 2^{\deg (s)} \sum_{\substack{q \geqslant Q, \\ s|q}}\frac{\tau(q)^{4}}{ 2^{2\deg( q )}} \\ \ll \tau(r)^{3} \sum_{s|r} \sum_{q \geqslant Q} \frac{\tau(q)^{4}}{2^{2\deg(q)}} = \tau(r)^{4} \sum_{q \geqslant Q} \frac{\tau(q)^{4}}{2^{2\deg(q)}}.
\end{multline*}

From the divisor bound for polynomials, Lemma \ref{divisor_bound}, for any $\eps'>0$ we have $\tau(q) \ll_{\eps'} 2^{\eps'\deg (q)}$. Hence, $D \lessapprox_{\eps'} \tau(r)^{4}\sum_{q \geqslant Q}2^{\eps'-2}$. Since
$$\sum_{q \in G_N} 2^{(\eps'-2)\deg (q)}=\sum_{n < N}2^{n-1} \cdot 2^{(\eps'-2)n}=\frac{1}{2}\sum_{n<N}2^{(\eps'-1)n}=\frac{1}{2} \cdot \frac{1-2^{(\eps'-1)N}}{1-2^{\eps'-1}}$$
and $\sum_{q \in \zff} 2^{(\eps'-2)\deg(q)} = \frac{1}{2} \cdot \frac{1}{1-2^{\eps'-1}}$, we have
$$D \lessapprox_{\eps'} \tau(r)^{4} \frac{2^{(\eps'-1)Q}}{1-2^{\eps'-1}}.$$
For $\eps'$ sufficiently small, this is $\ll_{\eps'} \tau(r)^{4} \cdot 2^{(\eps'-1)Q}$. Since $r=\denom(\lambda)$, $\deg(r) \leqslant K(1+\eps) +O_{\eps}(\log N)$ from the assumptions of the proposition and by the divisor bound, we get $D \lessapprox_{\eps'} 2^{K(1+\eps)\eps'+(\eps'-1)Q + O_{\eps}(\log N)}$. Since $K=Q-\eps$, for $\eps'$ sufficiently small we have $D \leqslant \frac{1}{5}2^{-K}$ for $N$ sufficiently large.

\begin{proof}[Proof of Lemma \ref{main_lemma_for_34}]
    Let us first notice that for each $a_1$ such that $\denom(\frac{b}{r}-\frac{a_1}{q_1})=q_2$, there exists exactly one $a_2 \in \zffmi{q_2}$ such that $\frac{a_1}{q_1}+\frac{a_2}{q_2} \equiv \frac{b}{r} \ (\modd \ \zff)$. Hence, $$S=\bigg| \sum_{\substack{a \in \zffmi{q_1}, \\ \denom(\frac{b}{r}-\frac{a_1}{q_1})=q_2}} e \Big( \frac{a_1}{q_1} \Big) \bigg|.$$
    Let us denote $(r, q_1)=d, \ r=r' \cdot d$ and $q_1 = q_1' \cdot d$. We have $\denom(\frac{b}{r}-\frac{a}{q_1})=q_2$ $\iff$ $(bq_1-ar, q_1r)=\frac{q_1r}{q_2}$ $\iff$ $(ar'-bq_1', dr'q_1')=\frac{dr'q_1'}{q_2}$. Since $r', q_1' \perp ar'-bq_1'$, for the last condition to hold we need to have $r'q_1' | q_2 | dr'q_1'$. Hence, we may write $q_2=r'q_1' \cdot t$ for $d=ts$. Then $\denom(\frac{b}{r}-\frac{a}{q_1})=q_2$ is equivalent to $(ar'-bq_1', d)=s$ and we have
    \begin{multline*}
        S
        = \bigg| \sum_{\substack{a \in \zffmi{q_1}, \\ (ar'-bq_1', d)=s}}  e \Big( \frac{a}{q_1} \Big) \bigg|
        = \bigg| \sum_{c \in \zffmi{\frac{d}{s}}} \sum_{\substack{a \in \zffmi{q_1}, \\ ar'-bq_1' \equiv cs (\modd \ d)}} e \Big( \frac{a}{q_1} \Big) \bigg| \\
        = \bigg| \sum_{c \in \zffmi{\frac{d}{s}}} \sum_{\substack{a \in \zffmi{q_1}}} e \Big( \frac{a}{q_1} \Big) \frac{1}{2^{\deg( d )}} \sum_{f \in \zff / (d)} e \Big( \frac{f(ar'-bq_1' - cs)}{d} \Big) \bigg| \\
        \leqslant \frac{1}{2^{\deg( d )}} \sum_{f \in \zff / (d)} \bigg| \sum_{a \in \zffmi{q_1}} e \Big( \frac{a(1+f r'q_1')}{q_1} \Big) \bigg| \cdot \bigg| \sum_{c \in \zffmi{\frac{d}{s}}} e \Big( \frac{-cf}{d/s} \Big) \bigg|.
    \end{multline*}
    From bounds on Ramanujan sums for polynomials (Lemma \ref{ramanujamsums}), the sum over $a$ in the last expression is at most $2^{\deg( (1+fr' q_1', q_1) )}=2^{\deg( (1+f r' q_1', d) )}$ and the sum over $c$ is at most $2^{\deg( (f, d/s) )} \leqslant 2^{\deg( (f, d) )}$. Hence, we have
    \begin{multline*}
        S
        \leqslant \frac{1}{2^{\deg( d )}} \sum_{f \in \zff / (d)} 2^{\deg( (1+f r' q_1', d)) + \deg( (f, d) )}
        = \frac{1}{2^{\deg( d )}} \sum_{f \in \zff / (d)} 2^{\deg( (f+f^2 r' q_1', d) )} \\
        \leqslant 2\frac{1}{2^{\deg( d )}} \sum_{f \in \zff / (d)} 2^{\deg( (f, d) )}
        \leqslant \frac{1}{2^{\deg( d )}} 2\sum_{l | d} 2^{\deg( l ) + \deg( \frac{d}{l} )}
        = 2\tau(d) \leqslant 2\tau(r) \ll \tau(r),
    \end{multline*}
    where the first equality holds as $(f, 1+f r' q_1')=1$ and the next inequality holds since the latter sum contains all terms the former one contains, and all summands are positive.
\end{proof}

\section{Proof of Proposition \ref{original_3_6}}
\label{section_3_6}

We need to show that the Fourier coefficients $\beta(\lambda)$ of
$$\widetilde{\Lambda}_R(f+1)\widetilde{H}_R(f) = \prod_{\substack{r \in G_R, \\ r \pprime}} \Lambda_r(f+1) \cdot \prod_{\substack{r \in G_R, \\ r \pprime}} \tau^2_r(f)$$
are real and non-negative, and that $\beta(0) \gg 1$.

For $r$ prime, let us define
\begin{equation}
    u_r(f)=\Big(1-\frac{1}{2^{\deg(r)}}\Big)\Big(1+\frac{3}{2^{\deg(r)}}\Big)\Lambda_r(f) \tau^2_r(f) = \begin{cases}
    0 & \text{if } f \equiv -1 \ (\modd \ r) \\
    4 & \text{if } f \equiv 0 \ (\modd \ r) \\
    1 & \text{otherwise} \\
\end{cases}
\end{equation}
and note that $$\widehat{u}_r(g) = \begin{cases}
    1 + \frac{2}{2^{\deg(r)}} & \text{if } g \equiv 0 \  (\modd \ r) \\
    \frac{1}{2^{\deg(r)}}(3-e(gr^{-1})) & \text{otherwise}
\end{cases}.$$
Hence, since $$\widetilde{\Lambda}_R(f)\widetilde{H}_R(f)=\prod_{\substack{r \in G_R, \\ r \pprime}}\Big(1 - \frac{1}{2^{\deg(r)}}\Big)^{-1}\Big( 1 + \frac{3}{2^{\deg(r)}} \Big)^{-1}u_r(f),$$
for $s$ squarefree with all prime factors of degree smaller than $R$ and $t \in \zffmodinvs$, we have $$\beta\Big(\frac{t}{s}\Big) = \prod_{\substack{r \in G_R, \\ r \pprime}} \Big(1 - \frac{1}{2^{\deg(r)}}\Big)^{-1}\Big( 1 + \frac{3}{2^{\deg(r)}} \Big)^{-1} \prod_{\substack{r \in G_R, \\ r \nmid s, \ r \pprime}} \Big( 1 + \frac{2}{2^{\deg(r)}} \Big) \prod_{\substack{r \in G_R, \\ r \mid s, \ r \pprime}} \frac{1}{2^{\deg(r)}} \big( 3 - e(f(r)r^{-1}) \big)$$
where $f(r)$ are such that for $s=r_1 \dots r_m$, we have $\frac{f(r_1)}{r_1}+ \dots + \frac{f(r_m)}{r_m}=\frac{t}{s}$. From this we can clearly see that the coefficients $\beta(\frac{t}{s})$ are non-negative.

Moreover, $$\beta(0)=\prod_{\substack{r \in G_R, \\ r \pprime}} \Big(1 - \frac{1}{2^{\deg(r)}}\Big)^{-1}\Big( 1 + \frac{3}{2^{\deg(r)}} \Big)^{-1} \Big( 1 + \frac{2}{2^{\deg(r)}} \Big)=\prod_{\substack{r \in G_R, \\ r \pprime}} \Big(1 + \frac{3}{2^{2\deg(r)} + 2 \cdot 2^{\deg(r)} - 3} \Big) \gg 1.$$
This finishes the proof of Proposition \ref{original_3_6}, and so all propositions from Section \ref{main_section} have now been proved. Together with considerations in that section, this finishes the proof of Theorem \ref{original_2_1} and so the proof of Theorem \ref{original_1_1}.

\appendix

\section{Some properties of function fields, analogous to the integers}
\label{section_appendix}

The results below hold for any $q$ a prime power and are classical, with proofs analogous to the case of integers.

\begin{lemma}[Dirichlet approximation for polynomials, \cite{dirichlet_polynomials}]
\label{dirichlet}
    Let $n \in \mathbf{N}$ and $\theta \in \mathbf{K}_{\infty}$ (for $\mathbf{K}_{\infty}$ as in (\ref{k_infty})). There exist $s, u \in \mathbf{F}_q[x]$ such that $\deg(s) < n$ and $\ord(\theta-\frac{u}{s}) \leqslant -n-\deg(s)$.
\end{lemma}

\begin{lemma}
\label{divisor_bound}
    For any $\eps>0$, $f \in \mathbf{F}_q[x]$, we have $$\tau(f) \ll_{\eps} q^{ \eps \cdot \deg(f)}.$$
\end{lemma}

\begin{lemma}
\label{original_A_1}
    We have $$\sum_{\substack{f \in \mathbf{F}_q[x]: \\ \deg(f) < N}} \tau(f)^B \leqslant q^N \cdot N^{O_B(1)}.$$
\end{lemma}

\begin{lemma}
\label{ramanujamsums}
    For any $h, c \in \mathbf{F}_q[x]$, we have the following bound for the Ramanujan sums over polynomials: $$\bigg| \sum_{a \in \zffmi{h}}e\Big( \frac{ac}{h} \Big) \bigg| \leqslant q^{\deg( (c, h) )}.$$
\end{lemma}

\end{document}